\documentclass[graybox]{svmult}

\usepackage{amsfonts,amsmath,amstext,latexsym,amssymb,physics}
\usepackage{mathptmx}       % selects Times Roman as basic font
\usepackage{helvet}         % selects Helvetica as sans-serif font
\usepackage{courier}        % selects Courier as typewriter font
\usepackage{type1cm}        % activate if the above 3 fonts are
\usepackage{makeidx}         % allows index generation
\usepackage{graphicx}        % standard LaTeX graphics tool
\usepackage{multicol}        % used for the two-column index
\usepackage[bottom]{footmisc}% places footnotes at page bottom
\makeindex             % used for the subject index
\newtheorem{cor}{Corollary}[section]
\newtheorem{exam}{Example}[section]
\newtheorem{rem}{Remark}[section]

\begin{document}

\title*{Ore extensions of functional algebras}
\author{Alex Behakanira Tumwesigye, Johan Richter, Sergei Silvestrov }
\institute{Alex Behakanira Tumwesigye \at Department of Mathematics, College of Natural Sciences, Makerere University, Box 7062, Kampala, \email{alexbt@cns.mak.ac.ug}
\and Johan Richter \at Division of Applied Mathematics, UKK, M\"{a}lardalens University, Box 883, V\"{a}ster\.{a}s  \email{johan.richter@mdh.se}
\and Sergei Silvestrov \at Division of Applied Mathematics, UKK, M\"{a}lardalens University, Box 883, V\"{a}ster\.{a}s  \email{sergei.silvestrov@mdh.se}}
%\and Vincent Ssembatya \at Department of Mathematics, College of Natural Sciences, Makerere University, Box 7062, Kampala, \email{vas@qad.mak.ac.ug}}
\maketitle

\abstract*{In this article we consider the  Ore extension algebra for the algebra  $\mathcal{A}$ of functions with finite support on a countable set. We derive explicit formulas for twisted derivations on $\mathcal{A},$ give a description for the centralizer of $\mathcal{A},$ and the center of the Ore extension algebra under specific conditions.}

\abstract{In this article we consider the  Ore extension algebra for the algebra  $\mathcal{A}$ of functions with finite support on a countable set. We derive explicit formulas for twisted derivations on $\mathcal{A},$ give a description for the centralizer of $\mathcal{A},$ and the center of the Ore extension algebra under specific conditions.}
%\documentclass[12pt]{article}

%\begin{document}

%\begin{center}
%\title{Description of commutants in crossed product algebras for piece-wise constant functions.}
%\author{Alex Behakanira Tumwesigye\thanks{Department of Mathematics, Makerere university, Box 7062, Kampala} \and{Sergei Silvestrov \thanks{Division of applied mathematics, UKK, M\"{a}lardalens, Box 883, V\"{a}ster\.{a}s }} \and{Johan Richter \thanks{Division of applied mathematics, UKK, M\"{a}lardalens, Box 883, V\"{a}ster\.{a}s }}\and{Vincent Ssembatya\thanks{Department of Mathematics, Makerere university, Box 7062, Kampala}}}
%\maketitle
%\end{center}
%\keywords{<Enter Keywords here>}

%\author{}{Tumwesigye Behakanira Alex \\
 % {Makerere University-Uganda,\\ M$\ddot{a}$lardalens University-Sweden}}

%\author{<author2>}{
 % Sergei Silvestrov\\{ M$\ddot{a}$lardalens University-Sweden}
%}
\section{Introduction}
A topic of interest in the field of operator algebras is the connection between properties of dynamical systems and algebraic properties of crossed products associated to them. More specifically the question when  
a certain canonical subalgebra is maximal commutative and has the ideal intersection property, that is, each non-zero ideal of the algebra intersects the subalgebra non-trivially. For a topological dynamical systems $(X,\sigma),$ 
one may define a crossed product C*-algebra $C(X) \rtimes_{\tilde{\sigma}} \mathbb{Z} $ where 
$\tilde{\sigma}$ is an automorphism of $C(X)$ induced by $\sigma$. It turns out that the property known as topological freeness of the dynamical system is equivalent to $C(X)$ being a maximal commutative subalgebra of 
$C(X) \rtimes_{\tilde{\sigma}} \mathbb{Z}$ and also equivalent to the condition that every non-trivial closed ideal has a non-zero intersection with $C(X)$. An excellent reference for this correspondence is \cite{TomiyamaABT}.
For analogues, extensions and applications of this theory in the study of dynamical systems, harmonic analysis, quantum field theory, string theory, integrable systems, fractals and wavelets, see \cite{ArchboldABT,CarlsenABT,Dutkay1ABT,Dutkay2ABT,ExelABT,
MackeyABT,OstroABT,TomiyamaABT}.

For any class of graded rings, including gradings given by semigroups or even filtered rings (e.g. Ore extensions), it makes sense to ask whether the ideal intersection property is related to maximal commutativity of the degree zero component. 
For crossed product-like structures, where one has a natural action, it further makes sense to ask how the above mentioned properties of the degree zero component are related to properties of the action. These questions have been considered recently for algebraic crossed products and Banach algebra crossed products, both in the traditional context of crossed products by groups as well as generalizations to crossed products by groupoids and general categories \cite{DeABT,LundABT,Alex1ABT,AlexABT,SvenssonABT,Silvestrov1ABT,Oinert1ABT,OinertABT}.

Ore extensions constitute an important class of rings, appearing in extensions of differential calculus, in non-commutative geometry, in quantum groups and algebras and as a uniting framework for many algebras
appearing in physics and engineering models. An Ore extension of a ring $R$ is an overring with a generator $x$ satisfying $xr=\sigma(r)x+\Delta(r), \ r \in R,$ for some endomorphism $\sigma$ and a 
 $\sigma$-derivation $\Delta$. \par 
 This article aims at giving a description of the centralizer and the center of the coefficient subalgebra $\mathcal{A}$ in the Ore extension algebra $\mathcal{A}[x,\tilde{\sigma},\Delta],$ where $\mathcal{A}$ is the algebra of functions with finite support on a countable set $X$  and $\tilde{\sigma}:\mathcal{A}\to \mathcal{A}$ is an automorphism of $\mathcal{A}$ that is induced by a bijection $\sigma:X\to X$. A number of studies on centralizers in Ore extensions have been carried out before in \cite{RichterABT,Richter2ABT,Richter1ABT}, but this is in completely different settings the one here.\par 
 In section \ref{OresecTumwesigye}, we recall some notation and basic facts about Ore extensions used through out the rest of the article. In section \ref{secderivationsTumwesigye}, we give a description of twisted derivations on the algebra of functions on a finite set from which it is observed that there are no non trivial derivations on $\mathbb{R}^n.$ In section \ref{seccentralTumwesigye}, we give the description of the centralizer of the coefficient algebra $\mathcal{A}$ and the center of the Ore extension $\mathcal{A}[x,\tilde{\sigma},0].$ In section \ref{sec5Tumwesigye}, we turn to the case when $\mathcal{A}$ is the algebra of functions on a countable set with finite support, give a description for the centralizer and the center of the Ore extension $\mathcal{A}[x,\tilde{\sigma},0].$ Sections \ref{sec6Tumwesigye} and \ref{sec7Tumwesigye} are devoted to the special case of the skew power series and the skew Laurent rings respectively.
 \section{Ore extensions. Basic preliminaries}\label{OresecTumwesigye}
 For general references on Ore extensions, see for example, \cite{GoodearlABT,RowenABT}. For the convenience of the reader, we recall the definition.\par 
  Let $R$ be a ring , 
$\sigma : R \to R$ a ring endomorphism (not necessarily injective) and $\Delta : R \to R$ a $\sigma$-derivation, that is,
\begin{displaymath}
	\Delta(a+b)=\Delta(a)+\Delta(b)
	\quad \text{ and } \quad
	\Delta(ab) = \sigma(a) \Delta(b) + \Delta(a)b
\end{displaymath}
for all $a,b\in R$.

\begin{definition}
The Ore extension $R[x,\sigma,\Delta]$ is defined as the ring generated by $R$ and an element $x \notin R$ such that $1,x, x^2, \ldots$ form a basis for $R[x,\sigma,\Delta]$ as a left $R$-module and all $r \in R$ satisfy
\begin{equation}\label{Eqn1Tumwesigye}
	x r = \sigma(r)x + \Delta(r).
\end{equation}
\end{definition}
Such a ring always exists and is unique up to isomorphism \cite{GoodearlABT}.\par  
From $\Delta(1 \cdot 1) = \sigma(1) \cdot 1 +\Delta(1) \cdot 1$ we get that $\Delta(1)=0$, and since $\sigma(1)=1$ we see that $1_R$ will be the multiplicative identity for $R[x,\sigma,\Delta]$ as well. 

If $\sigma= id_R$, then we say that $R[x,id_R,\Delta]$ is a \emph{differential polynomial ring}.
 If instead $\Delta \equiv 0$, then we say that $R[x,\sigma,0]$ is a 
\emph{skew polynomial ring}. The reader should be aware that some authors use the term \emph{skew polynomial ring} to mean  Ore extensions.

An arbitrary non-zero element $P \in R[x,\sigma,\Delta]$ can be written uniquely
as $P = \sum\limits_{i=0}^n a_i x^i$ for some $n \in \mathbb{Z}_{\geq 0}$, with $a_i \in R$ for $i \in \{0,1,\ldots, n\}$ and $a_n \neq 0$. The \emph{degree} of $P$ will be defined as; $\deg(P):=n$. We set $\deg(0) := -\infty$.

\begin{definition}
 A $\sigma$-derivation $\Delta$ is said to be \emph{inner} if there exists some $a \in R$ such that $\Delta(r) = ar-\sigma(r)a$ for all $r \in R$. A $\sigma$-derivation that is not inner is called \emph{outer}.
\end{definition}

%\noindent {\bf Notation:}
Given a ring $S$ we denote its center by $Z(S)$.
The \emph{centralizer} $C(T),$ of a subset $T\subseteq S$ is defined as the set of elements of $S$ that commute with every element of $T$. If $T$ is a commutative subring of $S$ and the centralizer of $T$ in $S$ coincides with $T$, 
then $T$ is said to be a \emph{maximal commutative} subring of $S$.
\section{Derivations on algebras of functions on a finite set}\label{secderivationsTumwesigye}
Let $X=[n](=\{1,2,\cdots,n\})$ be a finite set and let $\mathcal{A}=\qty{f:X\to \mathbb{R}}$ denote the algebra of real-valued functions on $X$ with respect to the usual pointwise operations, that is, pointwise addition, scalar multiplication and pointwise multiplication. By writing $f_n:=f(n)$, then we can identify $\mathcal{A}$ with $\mathbb{R}^n.$ Here, $\mathbb{R}^n$ is equipped with the usual operations of pointwise addition, scalar multiplication and multiplication defined by
$$xy=(x_1y_1,x_2y_2,\cdots,x_ny_n)$$ for every $x=(x_1,x_2,\cdots,x_n)$ and $y=(y_1,y_2,\cdots,y_n).$\par 
Now, let $\sigma :X\to X$ be a bijection such that $\mathcal{A}$ is invariant under $\sigma$, (that is $\sigma$ is a permutation on $X$) and let $\tilde{\sigma}:\mathcal{A}\to \mathcal{A}$ be the automorphism induced by $\sigma,$ that is 
\begin{equation}\label{Eqn2Tumwesigye}
\tilde{\sigma}(f)=f\circ \sigma^{-1}
\end{equation}
for every $f\in \mathcal{A}.$ We would like to consider the Ore extension $\mathcal{A}[x;\tilde{\sigma},\Delta]$ where $\Delta$ is a $\tilde{\sigma}-$ derivation on $\mathcal{A}$ and $x$ is an indeterminate.\par 
Recall that $\Delta$ is a $\tilde{\sigma}$ derivation on $\mathcal{A}$ if it is $\mathbb{R}-$linear and for every $f,g\in \mathcal{A},$
$$
\Delta(fg)=\tilde{\sigma}(f)\Delta(g)+\Delta(f)g.
$$
Since $\mathcal{A}$ can be identified with $\mathbb{R}^n$, then $\Delta$ is an operator on $\mathbb{R}^n$ which can be represented by a matrix, that is; 
\begin{equation}\label{Eqn3Tumwesigye}
\qty[\Delta]=\qty[\Delta(e_1)|\Delta(e_2)|\cdots |\Delta(e_n)],
\end{equation}
where $\qty{e_1,e_2,\cdots,e_n}$ is the standard basis of $\mathbb{R}^n.$ In the following Theorem we give the description of the matrix $[\Delta]$ in \eqref{Eqn3Tumwesigye} above.
\begin{theorem}\label{thm1Tumwesigye}
Let $\sigma:X\to X$ be a bijection and let $\Delta$ be a $\tilde{\sigma}-$derivation whose standard matrix $[\Delta]=\qty[k_{ij}]$ is as given by \eqref{Eqn3Tumwesigye}. Then 
\begin{enumerate}
\item $k_{li}=0$ if $l\not\in \{i,\sigma(i)\}$,
\item $k_{ji}=-k_{jj}$ for all $i\neq j.$
\end{enumerate}
\end{theorem}
\begin{proof}
\begin{enumerate}
\item If $\sigma(i)=j$, then $\tilde{\sigma}(e_i)=e_j,$ where $\tilde{\sigma}$ is as defined by \eqref{Eqn2Tumwesigye} and $\{e_i,\ i=1,2,\cdots,n\}$ is the standard basis for $\mathbb{R}^n.$ Therefore from the definition of $\Delta,$
\begin{align*}
\Delta(e_i^2)&=\tilde{\sigma}(e_i)\Delta(e_i)+\Delta(e_i)e_i\\
&=e_j\Delta(e_i)+\Delta(e_i)e_i\\
&=\Delta(e_i)(e_j+e_i).
\end{align*}
Now $e_i^2=e_i$ and hence 
$$\Delta(e_i)=\Delta(e_i)(e_i+e_j).$$
The  term on the right side is zero except in the $i^{th}$ and $j^{th}$ component. Therefore, 
$$\Delta(e_i)=\mqty[k_{1i}\\ k_{2i}\\ \vdots\\ k_{ni}]$$
where $k_{li}=0$ whenever $l\neq i,j.$
\item For $i\neq j$, $\Delta(e_ie_j)=\Delta(0)=0$ and hence
\begin{align*}
0&=\Delta(e_ie_j)\\
&=\tilde{\sigma}(e_i)\Delta(e_j)+\Delta(e_i)e_j\\
&=e_j\Delta(e_j)+\Delta(e_i)e_j\\
&=e_j\qty(\Delta(e_j)+\Delta(e_i)),
\end{align*}
where the right hand side term is zero except in the $j^{th}$ component. Looking at the $j^{th}$ component we get
$$k_{jj}+k_{ji}=0 \ \text{ or }\ k_{ji}=-k_{jj}.$$
\end{enumerate}
\end{proof}{\flushright{\qed}}
\begin{cor}
There are no non zero derivations on $\mathcal{A}.$
\end{cor}
\begin{proof}
If $\tilde{\sigma}=id,$ that is, $\tilde{\sigma}(e_i)=e_i,\ i=1,2,\cdots,n,$ then from Theorem \ref{thm1Tumwesigye}, it follows that if $[\Delta]=[k_{ij}],$ then $k_{ij}=0$ for all $i\neq j$, and from $k_{ii}=-k_{ij}$ we have $k_{ii}=0$ for all $i=1,2,\cdots, n$. Therefore, $\Delta=0.$
\end{proof}{\flushright{\qed}}
\begin{exam}
\end{exam}
Let $n=3$ and let $\sigma:[3]\to [3]$ be a permutation such that $\sigma(1)=2,\ \sigma(2)=3$ and $\sigma(3)=1.$ Then 
\begin{align*}
\tilde{\sigma}(f)(1)&=f\qty(\sigma^{-1}(1))=f(3)\\
\tilde{\sigma}(f)(2)&=f\qty(\sigma^{-1}(2))=f(1)\\
\tilde{\sigma}(f)(3)&=f\qty(\sigma^{-1}(3))=f(2)
\end{align*}
Therefore $\tilde{\sigma}(x_1,x_2,x_3)=(x_3,x_1,x_2).$
Let 
$$
[\Delta]=\mqty[k_{11}&k_{12}&k_{13}\\ k_{21}&k_{22}&k_{23}\\ k_{31}&k_{32}&k_{33}] 
$$
be the standard matrix for $\Delta$ and let $x=(x_1,x_2,x_3)$ and $y=(y_1,y_2,y_3)$ be arbitrary vectors in $\mathbb{R}^3.$ Then
\begin{align*}
\tilde{\sigma}(x)\Delta(y)&=\mqty[x_3\\x_1\\x_2]\mqty[k_{11}y_1&k_{12}y_2&k_{13}y_3\\ k_{21}y_1&k_{22}y_2&k_{23}y_3\\ k_{31}y_1&k_{32}y_2&k_{33}y_3]\\
&=\mqty[k_{11}x_3y_1&k_{12}x_3y_2&k_{13}x_3y_3\\ k_{21}x_1y_1&k_{22}x_1y_2&k_{23}x_1y_3\\ k_{31}x_2y_1&k_{32}x_2y_2&k_{33}x_2y_3]
\end{align*} 
\begin{align*}
\Delta(x)y&=\mqty[k_{11}x_1&k_{12}x_2&k_{13}x_3\\ k_{21}x_1&k_{22}x_2&k_{23}x_3\\ k_{31}x_1&k_{32}x_2&k_{33}x_3]\mqty[y_1\\ y_2\\ y_3]\\
&=\mqty[k_{11}x_1y_1&k_{12}x_2y_1&k_{13}x_3y_1\\ k_{21}x_1y_2&k_{22}x_2y_2&k_{23}x_3y_2\\ k_{31}x_1y_3&k_{32}x_xy_3&k_{33}x_3y_3]
\end{align*}
$$
\tilde{\sigma}(x)\Delta(y)+\Delta(x)y=\mqty[k_{11}(x_1y_1+x_3y_1)&k_{12}(x_2y_1+x_3y_2)&k_{13}(x_3y_1+x_3y_3)\\ k_{21}(x_1y_2+x_1y_1)&k_{22}(x_1y_2+x_2y_2)&k_{23}(x_1y_3+x_3y_2)\\ k_{31}(x_2y_1+x_1y_3)&k_{32}(x_2y_2+x_2y_3)&k_{33}(x_2y_3+x_3y_3)].
$$
Also,
\begin{align*}
\Delta(xy)&=\mqty[k_{11}&k_{12}&k_{13}\\ k_{21}&k_{22}&k_{23}\\ k_{31}&k_{32}&k_{33}]\mqty[x_1y_1\\ x_2y_2\\ x_3y_3]\\
&=\mqty[k_{11}x_1y_1&k_{12}x_2y_2&k_{13}x_2y_3\\ k_{21}x_1y_1&k_{22}x_2y_2&k_{23}x_3y_3\\ k_{31}x_1y_1&k_{32}x_2y_2&k_{33}x_3y_3].
\end{align*}
Therefore, from 
$$\Delta(xy)=\tilde{\sigma}(x)\Delta(y)+\Delta(x)y$$ we get
$$k_{11}x_3y_1+k_{12}(x_2y_1+x_3y_3-x_2y_2)+k_{13}x_3y_1=0,$$
from which we obtain, $k_{13}=-k_{11}$ and $k_{12}=0.$\par 
Similarly, we obtain $k_{21}=-k_{22},\ k_{23}=0$, and $k_{31}=0,\ k_{32}=-k_{33}$ which is agreement with the assertions of Theorem \ref{thm1Tumwesigye}.\par 
Setting $k_{11}=s,\ k_{22}=t$ and $k_{33}=u$ we obtain the matrix of $\Delta$ as;
$$[\Delta]=\mqty[s&0&-s\\ -t &t&0\\ 0&-u&u].$$
\par 
In the next Theorem, we prove that if $\tilde{\sigma}:\mathbb{R}^n\to \mathbb{R}^n$ is an automorphism of $\mathbb{R}^n$ and $\Delta:\mathbb{R}^n\to \mathbb{R}^n$ is an operator on $\mathbb{R}^n$ that satisfies the conditions of Theorem \ref{thm1Tumwesigye}, then $\Delta$ is a $\tilde{\sigma}-$ derivation.
\begin{theorem}\label{thm2Tumwesigye}
Let $\sigma:X\to X$ be a bijection on $X$ and let $\tilde{\sigma}:\mathbb{R}^n\to \mathbb{R}^n$ be the automorphism induced by $\sigma.$ Let $\Delta:\mathbb{R}^n\to \mathbb{R}^n$ be an operator whose standard matrix $\qty[\Delta]=\qty[k_{ij}]$ has the following properties
\begin{enumerate}
\item $k_{li}=0$ if $l\not\in\{ i,\sigma(i)\}$,
\item $k_{ji}=-k_{jj}.$
\end{enumerate}
Then $\Delta$ is a $\tilde{\sigma}-derivation.$
\end{theorem}
\begin{proof}
We first do the proof for the standard basis $\qty{e_1,e_2,\cdots,e_n}$ of $\mathbb{R}^n.$ Recall that if $\sigma(i)=j$ then $\tilde{\sigma}(e_i)=e_j.$ From the definition of $\Delta,$
$$
\Delta(e_i)=\mqty[k_{1i}\\ k_{2i}\\ \vdots\\ k_{ni}]
$$
where $k_{li}=0$ for all $l\neq i,j$ and $k_{ji}=-k_{jj}.$\par 
Now, for $i\neq j,$ $\Delta(e_ie_j)=\Delta(0)=0.$ And
\begin{align*}
\tilde{\sigma}(e_i)\Delta(e_j)+\Delta(e_i)e_j&=e_j\Delta(e_j)+\Delta(e_i)e_j\\
&=e_j\qty(\Delta(e_j)+\Delta(e_i)).
\end{align*}
All the components in the above vector are zero except the $j^{th}$ component which is given by
$$k_{jj}+k_{ji}=k_{jj}-k_{jj}=0.$$
Also, $\Delta(e_i^2)=\Delta(e_i)$, (since $e_i^2=e_i$), and 
\begin{align*}
\tilde{\sigma}(e_i)\Delta(e_i)+\Delta(e_i)e_i&=e_j\Delta(e_i)+\Delta(e_i)e_i\\
&=\Delta(e_i)(e_i+e_j),
\end{align*}
where all the components in the above vector are zero except the $i^{th}$ and $j^{th}$ component. That is (assuming $i<j$)
$$\tilde{\sigma}(e_i)\Delta(e_i)+\Delta(e_i)e_i=\mqty[0\\ \vdots\\ k_{ii}\\0\\ \vdots\\ k_{ji}\\0\\ \vdots\\0]=\Delta(e_i).$$
Therefore, $\Delta$ is a $\tilde{\sigma}-$derivation on the standard basis vectors.\par 
Now let $x=(x_1,x_2,\cdots,x_n)$ and $y=(y_1,y_2,\cdots,y_n)$ be arbitrary vectors in $\mathbb{R}^n.$ Then 
$$x=\sum_{i=1}^nx_ie_i\ \text{ and }\ y=\sum_{j=1}^ny_je_j$$
and $$xy=\sum_{i,j=1}^nx_iy_je_ie_j.$$

Using the fact that both $\tilde{\sigma}$ and $\Delta$ are $\mathbb{R}-$linear we have;
\begin{align*}
\Delta (xy)&=\Delta\qty(\sum_{i,j=1}^nx_iy_je_ie_j)\\
&=\sum_{i,j=1}^n\Delta\qty(x_iy_je_ie_j)\\
&=\sum_{i,j=1}^nx_iy_j\Delta\qty(e_ie_j)\\
&=\sum_{i,j=1}^n x_iy_j\qty(\tilde{\sigma}(e_i)\Delta(e_j)+\Delta(e_i)e_j)\\
&=\sum_{i,j=1}^n x_iy_j\tilde{\sigma}(e_i)\Delta(e_j)+\sum_{i,j=1}^nx_iy_j\Delta(e_i)e_j\\
&=\qty(\sum_{i=1}^nx_i\tilde{\sigma}(e_i))\qty(\sum_{j=1}^n y_j\Delta(e_j))+\qty(\sum_{i=1}^nx_i\Delta(e_i))\qty(\sum_{j=1}^ny_je_j)\\
&=\qty(\tilde{\sigma}\qty(\sum_{i=1}^nx_ie_i))\qty(\Delta\qty(\sum_{j=1}^n y_je_j))+\qty(\Delta\qty(\sum_{i=1}^nx_ie_i))\qty(\sum_{j=1}^ny_je_j)\\
&= \tilde{\sigma}(x)\Delta(y)+\Delta(x)y.
\end{align*}
Therefore $\Delta$ is a $\tilde{\sigma}-$derivation on $\mathbb{R}^n.$
\end{proof} {\flushright{\qed}}
\par 
\section{Centralizers in Ore extensions for functional algebras}\label{seccentralTumwesigye}
Consider the Ore extension $\mathcal{A}[x,\tilde{\sigma},\Delta]$, that is, the algebra
$$\mathcal{A}[x,\tilde{\sigma},\Delta]:=\qty{\sum_{k=0}^m f_kx^k,\ :\ f_k\in \mathcal{A}}$$
with the operations of pointwise addition, scalar multiplication, and multiplication given by the relation
$$xf=\tilde{\sigma}(f)x+\Delta (f)$$
for every $f\in \mathcal{A}.$ \par 
Our interest is to give a description of the centralizer $C(\mathcal{A})$, of $\mathcal{A}$ in the Ore extension $\mathcal{A}[x,\tilde{\sigma},\Delta]$ where $\tilde{\sigma}$ and $\Delta$ are as described before. Using the notation introduced in \cite{Richter1ABT}, we define functions $\pi_k^l:\mathcal{A}\to \mathcal{A},$ for $k,l\in \mathbb{Z}$ as follows; $\pi_0^0=id.$ If $m,n$ are nonzero integers such that $m>n$ or atleast one of the integers is negative, then $\pi_m^n=0.$ For the other remaining cases,
$$\pi_m^n=\tilde{\sigma}\circ \pi_{m-1}^{n-1}+\Delta \circ\pi_m^{n-1}.$$ 
It has been proven in \cite{Richter1ABT} that an element $\sum\limits_{k=0}^m f_kx^k\in \mathcal{A}[x,\tilde{\sigma},\Delta]$ belongs to the centralizer of $\mathcal{A}$ if and only if 
\begin{equation}\label{eqn4Tumwesigye}
gf_k=\sum_{j=k}^m f_j\pi_k^j(g)
\end{equation}
holds for all $k\in \{0,1,\cdots,m\}$ and all $g\in \mathcal{A}.$
\par 
Observe that since $\mathcal{A}$ is commutative, then the centralizer $C(\mathcal{A})$ of $\mathcal{A}$ is also commutative and hence a maximal commutative subalgebra of $\mathcal{A}[x,\tilde{\sigma},\Delta].$
\subsection{Centralizer for the case $\Delta=0$}
In this section we treat the simplest case when $\Delta=0.$ Recall that $\tilde{\sigma}$ acts like a permutation of the elements of $\mathcal{A}$ and since $[n]$ is a finite set, then $\tilde{\sigma}$ is of finite order.  Before we give the description, we need the following definition.
\begin{definition}\label{defn3Tumwesigye}
For any nonzero $n\in \Bbb Z,$ set
\begin{eqnarray*} 
%Sep_{\mathcal{A}}^n(X)&=&\{x\in X~|~\exists h\in \mathcal{A}~:~h(x)\neq \tilde{\sigma}^n(h)(x)\},\\
%Per_{\mathcal{A}}^n(X)&=&\{x\in X~|~\forall h\in \mathcal{A}~:~h(x)= \tilde{\sigma}^n(h)(x)\}, \\
Sep^n(X)&=&\{x\in X~|~x\neq \sigma^n(x)\}, \\
Per^n(X)&=&\{x\in X~|~x=\sigma^n(x)\}.
\end{eqnarray*}
\end{definition}
Observe that $\tilde{\sigma}^n(h)(x)\neq h(x)$ if and only if $\sigma^n(x)\neq x$ for every $x\in X$ and every $h\in \mathcal{A}.$ We give the description of the centralizer in the following theorem.
\begin{theorem}\label{thm3Tumwesigye}
The centralizer $C(\mathcal{A}),$ of $\mathcal{A}$ in the Ore extension $\mathcal{A}[x,\tilde{\sigma},\Delta]$ is given by
$$
C(\mathcal{A})=\qty{\sum_{k=0}^mf_kx^k\qq{ such that }f_k=0\qq{ on } Sep^k(X)}.
$$
\end{theorem}
\begin{proof}
From equation \eqref{eqn4Tumwesigye} and the fact that $\Delta=0,$  we see that an element $\sum\limits_{k=0}^m f_kx^k\in \mathcal{A}[x,\tilde{\sigma},\Delta]$ belongs to the centralizer of $\mathcal{A}$ if and only if 
\[
gf_k=f_k\tilde{\sigma}^k(g)
\] 
for every $k=0,1,\cdots,m$ and every $g\in \mathcal{A}$. That is,
\begin{equation}\label{eqn5Tumwesigye}
g(x)f_k(x)=f_k(x)\tilde{\sigma}^k(g)(x)
\end{equation}
for every $x\in X.$ Since $\mathcal{A}$ is commutative, then equation \eqref{eqn5Tumwesigye} will hold if and only if $x\in Per^k(X)$ or $f_k=0.$ Therefore, the centralizer $C(\mathcal{A}),$ of $\mathcal{A}$ will be given by
$$C(\mathcal{A})=\qty{\sum_{k=0}^mf_kx^k\ :\ f_k\in \mathcal{A} \text{ where } f_k=0 \text{ on } Sep^k(X)}.$$
\end{proof}{\flushright{\qed}}
\subsection{Centralizer for the case $\Delta\neq 0$}
Now, suppose $\tilde{\sigma}\neq id$ is of order $j\in \mathbb{Z}_{>0},$ that is $\tilde{\sigma}^j=id$ but $\tilde{\sigma}^k\neq id$ for all $k<j$. In the next Theorem, we state a necessary condition for an element in the Ore extension to belong to the centralizer.
\begin{theorem}\label{thm4Tumwesigye}
If an element of degree $m,$ $\sum\limits_{k=0}^mf_kx^k\in \mathcal{A}[x,\tilde{\sigma},\Delta]$ belongs to the centralizer of $\mathcal{A}$, then $f_m=0 \text{ on } Sep^m(X).$
\end{theorem}
\begin{proof}
As already stated an element $\sum\limits_{k=0}^mf_kx^k\in \mathcal{A}[x,\tilde{\sigma},\Delta]$ belongs to the centralizer of $\mathcal{A}$ if and only if for every $g\in\mathcal{A}$ 
$$gf_k=\sum_{j=k}^mf_j\pi_k^j(g)$$
for $k=0,1,\cdots,m.$ Looking at the leading coefficient we have 
\begin{equation}\label{eqn6Tumwesigye}
gf_m=f_m\pi_m^m(g)=f_m\tilde{\sigma}^m(g)
\end{equation}
Since $\mathcal{A}$ is commutative, then equation \eqref{eqn6Tumwesigye} holds on $ Per^m(X)$ and on $Sep^m(X),$ equation \eqref{eqn6Tumwesigye} holds if and only if $f_m=0.$
\end{proof} {\flushright{\qed}}\par 
The above condition is not sufficient to describe all the elements that belong to the centralizer of $\mathcal{A}$. In the next example we show that conditions satisfied by all elements in the centralizer of $\mathcal{A}$ are actually quite complicated even for the case when $n=2.$
\begin{exam}
\end{exam}
Let $n=2$ and let $\sigma:[n]\to [n]$ be a bijection on $[n].$ We already know that if $\sigma=id,$ then $\Delta=0$ so we will consider the case $\sigma\neq id,$ that is $\tilde{\sigma}(e_1)=e_2$ and $\tilde{\sigma}(e_2)=e_1.$ In this case, $\Delta$ has a standard matrix given by 
$$[\Delta]=\begin{bmatrix}
s&&-s\\ -t&& t
\end{bmatrix}$$
for some $s,t\in \mathbb{R}$ with $s,t\neq 0.$ 
\par A direct calculation shows that an element $fx\in \mathcal{A}[x,\tilde{\sigma},\Delta]$ belongs to the centralizer of $\mathcal{A}$ if and only if $f=0.$ So we consider a monomial of degree $2.$\par 
Let $f=\mqty(f_1\\ f_2)x^2\in \mathcal{A}[x,\tilde{\sigma},\Delta]$ be an element in the centralizer of $\mathcal{A}.$ Then, if $g=\mqty(g_1\\ g_2)\in \mathcal{A},$ we have,
$$gf=\mqty(g_1\\ g_2)\mqty(f_1\\ f_2)x^2=\mqty(g_1f_1\\ g_2f_2)x^2.$$  
On the other hand, using the fact that for every $g\in \mathcal{A,}$
$$x^2 g=\tilde{\sigma}^2(g)x^2+\qty[\Delta(\tilde{\sigma}(g))+\tilde{\sigma}\qty(\Delta(g))]x+\Delta^2(g)$$ and since $\tilde{\sigma}^2=id$ we have,
\begin{align*}
fg&=\mqty(f_1\\ f_2)\mqty(g_1\\ g_2)x^2+\mqty(f_1\\ f_2)\qty[\mqty(s&-s\\ -t&t)\mqty(g_2\\ g_1)+\tilde{\sigma}\qty(\mqty(s&-s\\ -t& t)\mqty(g_1\\g_2))]x\\
&\qquad +\mqty(f_1\\ f_2)\mqty(s^2+st & -(s^2+st)\\-(st+t^2)& st+t^2)\mqty(g_1\\g_2).\\
&=\mqty(f_1g_1\\ f_2g_2)x^2+\mqty[\mqty(-f_1(s+t)(g_1-g_2)\\ f_2(s+t)(g_1-g_2))]x+\mqty[f_1(s^2+st)(g_1-g_2)\\-f_2(st+t^2)(g_1-g_2)]
\end{align*}
Solving $fg=gf$ and looking at the coefficient of $x^2$, we get that $f_1,f_2$ are free variables and hence the centralizer $C(\mathcal{A})$ is non trivial. In the more general case, we have the following.
\par As already seen, an element $\sum\limits_{k=0}^mf_kx^k\in \mathcal{A}[x,\tilde{\sigma},\Delta]$ belongs to the centralizer of $\mathcal{A}$ if and only if for every $g\in\mathcal{A}$ 
$$gf_k=\sum_{j=k}^mf_j\pi_k^j(g)$$
for $k=0,1,\cdots,m.$\par 
 Looking at the constant term and using the fact that $\pi_0^j(g)=\Delta^j(g)$ we have;
\begin{eqnarray*}
gf_0&=&\sum_{j=0}^mf_j\pi_0^j(g)\\
&=&f_0g+\sum_{j=1}^mf_j\Delta^j(g)
\end{eqnarray*}
from which we obtain that 
\begin{equation}\label{eqn7Tumwesigye}
\sum_{j=1}^mf_j\Delta^j(g)=0.
\end{equation}
Now, for any $g\in \mathcal{A},\ g=\mqty(g_1\\ g_2),$ we have
$$
\Delta^j(g)=\Delta^j\mqty(g_1\\g_2)=\mqty[(g_1-g_2)\sum_{k=0}^j\mqty(j-1\\ k)s^{j-k}t^k\\ -(g_1-g_2)\sum_{k=0}^j\mqty(j-1\\ k)s^{k}t^{j-k}].
$$
Therefore, from equation \eqref{eqn7Tumwesigye}, we have 
$$
\mqty[(g_1-g_2)\sum_{j=1}^m f_{j1}\sum_{k=0}^j\mqty(j-1\\ k)s^{j-k}t^k\\ -(g_1-g_2)\sum_{j=0}^mf_{j2}\sum_{k=0}^j\mqty(j-1\\ k)s^{k}t^{j-k}]=\mqty[0\\0].
$$
Since equation \eqref{eqn7Tumwesigye} should hold for every $g\in \mathcal{A},$ then we have 
\begin{equation}\label{eqn8Tumwesigye}
\mqty[\sum_{j=1}^m f_{j1}\sum_{k=0}^j\mqty(j-1\\ k)s^{j-k}t^k\\ \sum_{j=0}^mf_{j2}\sum_{k=0}^j\mqty(j-1\\ k)s^{k}t^{j-k}]=\mqty[0\\0].
\end{equation}
Since $s,t\neq 0,$ then from equation \eqref{eqn8Tumwesigye}, we have 
$$
\mqty[\sum_{j=1}^m f_{j1}\sum_{k=0}^{j-1}\mqty(j-1\\ k)s^{j-k-1}t^k\\ \sum_{j=0}^mf_{j2}\sum_{k=0}^{j-1}\mqty(j-1\\ k)s^{k}t^{j-k-1}]=\mqty[0\\0].
$$
Observe that $$\sum\limits_{k=0}^j\mqty(j-1\\ k)s^{j-k-1}t^k=\sum\limits_{k=0}^j\mqty(j-1\\ k)s^{k}t^{j-k-1}.$$ Therefore we get a matrix equation of the form
\begin{equation}\label{eqn9Tumwesigye}
\mqty[f_{11}&f_{21}&f_{31}&\cdots&f_{m1}\\ f_{12}&f_{22}&f_{32}&\cdots&f_{m2}]\mqty[1\\ s+t\\ s^2+2st +t^2\\ \vdots\\ \sum_{k=0}^{m-1}\mqty(m-1\\ k)s^{m-k-1}t^t]
=\mqty[0\\0]\end{equation}
which always has nontrivial solutions if $m\geqslant2.$ 
\subsection{Center of the Ore extension algebra}
In the following section, we give a description of the center of our Ore extension algebra. We will give the description for the case $\Delta=0.$
\begin{theorem}\label{thm5Tumwesigye}
The center of the Ore extension algebra $\mathcal{A}[x,\tilde{\sigma},0]$ is given by
$$Z\qty(\mathcal{A}[x,\tilde{\sigma},0])=\qty{\sum_{k=0}^mf_kx^k:\text{ where }f_k=0 \text{ on } Sep^k(X) \text{ and }\tilde{\sigma}(f_k)=f_k}.$$
\end{theorem}
\begin{proof}
Let $f=\sum\limits_{k=0}^mf_kx^k$ be an element in $Z\qty(\mathcal{A}[x,\tilde{\sigma},0]),$ then $f\in C(\mathcal{A}),$ that is $f_k(x)=0$ for every $x\in Sep^k(x).$ Since the Ore extension $\mathcal{A}[x,\tilde{\sigma},0]$ is associative, it is enough to derive conditions under which $xf=fx.$ Now 
$$fx=\qty(\sum_{k=0}^mf_kx^k)x=\sum_{k=0}^mf_kx^{k+1}.$$ On the other hand,
\begin{align*}
xf&=x\sum_{k=0}^mf_kx^k\\
&=\sum_{k=0}^mxf_k x^k\\
&=\sum_{k=0}^m\tilde{\sigma}(f_k)x^{k+1}.
\end{align*}
From which we obtain that $xf=fx$ if and only if $\tilde{\sigma}(f_k)=f_k.$ Therefore,
$$Z\qty(\mathcal{A}[x,\tilde{\sigma},0])=\qty{\sum_{k=0}^mf_kx^k:\text{ where }f_k=0 \text{ on }Sep^k(x) \text{ and }\tilde{\sigma}(f_k)=f_k}.$$
\end{proof}{\flushright{\qed}}
%\section{Case when $\sigma$ is not invertible}
%In this section, we turn to our attention to the Ore extension $\mathcal{A}[x,\tilde{\sigma},\Delta]$ for non invertible $\tilde{\sigma}.$\par 
%Let $X=[n]=\{1,2,\cdots,n\}$ for some positive integer $n$ and as before, let $\mathcal{A}$ denote the algebra of functions $f:X\to X$ with the usual pointwise operations of addition, scalar multiplication and multiplication. Let $\sigma :X\to X$ be a finite-to-one map and let $\tilde{\sigma}:\mathcal{A}\to \mathcal{A}$ be the endomorphism of $\mathcal{A}$ induced by $\sigma,$ defined by
%\begin{equation}\label{end}
%\tilde{\sigma}(f)=f\circ \sigma
%\end{equation}
%Since $X$ is finite, there exists a positive integer $j$ such that $\sigma^{j+1}(X)=\sigma^j(X).$ Therefore, let $k$ be the smallest positive integer such that $\sigma^{k+1}(X)=\sigma^k(X).$ Then $X$ can be written as 
%$$X=\bigcup_{j=1}^kX_j$$ where 
%$$
%X_j:=\begin{cases} \sigma^{j-1}(X)\setminus \sigma^j(X)\ & \text{ if } j=1,2,\cdots,k-1\\
%\sigma^k(X) & \text{ if } j=k
%\end{cases}
%$$
%and the $X_j$ are disjoint. Using this, every $f\in \mathcal{A}$ can be written uniquely as 
%$$
%f(x) =\sum_{j=1}^nf_j(x)
%$$
%where $f_j(x)=f(x)\chi_{X_j}(x)$ for every $x\in X.$ For $j=1,2,\cdots, k$ define $I_j$ as follows;
%$$I_j:=\qty{f\in \mathcal{A}\ :\ f(x)=0 \text{ if } x\not\in X_j},$$
%Then $I_j$ is an ideal in $\mathcal{A},$ where, by an ideal, we mean a two sided ideal, since $\mathcal{A}$ is commutative. It follows that we have the following direct sum decomposition of $\mathcal{A}$  
%$$\mathcal{A}=\bigoplus_{j=1}^kI_j.$$
\section{Infinite dimensional case}\label{sec5Tumwesigye}
Let $J$ be a countable subset of $\mathbb{R}$ and let $\mathcal{A}$ be the set of functions $f:J\to J$ such that $f(i)=0$ for all except finitely many $i\in J.$ Then $\mathcal{A}$ is a commutative non-unital algebra with respect to the usual pointwise operations of addition, scalar multiplication and multiplication. For $i\in J,$ let $e_i\in \mathcal{A}$ denote the characteristic function of $i,$ that is 
$$e_i(j)=\chi_{\qty{i}}(j)=\begin{cases}1\ \text{ if }i=j\\ 0\ \text{ if }i\neq j.\end{cases}$$
Then every $f\in \mathcal{A}$ can be written in the form 
\begin{equation}\label{eqn10Tumwesigye}
f=\sum_{i\in J}f_ie_i
\end{equation}
where $f_i=0$ for all except finitely many $i\in J.$\par 
Let $\sigma:J\to J$ be a bijection and let $\tilde{\sigma}:\mathcal{A}\to \mathcal{A}$ be the automorphism of $\mathcal{A}$ induced by $\sigma,$ that is,
$$\tilde{\sigma}(f)=f\circ \sigma^{-1}$$ for every $f\in \mathcal{A}.$ We can still construct the non-unital Ore extension $\mathcal{A}[x,\tilde{\sigma},\Delta]$ as follows
$$
\mathcal{A}[x,\tilde{\sigma},\Delta]:=\qty{\sum_{k=0}^mf_kx^k\qq{ where }f_k\in \mathcal{A}}
$$
with addition and scalar multiplication given by the usual pointwise operations and multiplication determined by the relation
$$(fx)g=\tilde{\sigma}(g)fx+\Delta(g)$$ where $\Delta$ is a $\tilde{\sigma}-$derivation on $\mathcal{A}.$\par 
 In the following Theorem, we state the necessary and sufficient conditions for $\Delta:\mathcal{A}\to \mathcal{A}$ to be a $\tilde{\sigma}-$derivation on $\mathcal{A}.$
\begin{theorem}\label{thm6Tumwesigye}
Let $\sigma:J \to J$ be a bijection and let $\tilde{\sigma}:\mathcal{A}\to \mathcal{A}$ be the automorphism induced by $\sigma.$ A linear map $\Delta:\mathcal{A}\to \mathcal{A}$ is a $\tilde{\sigma}-$derivation on $\mathcal{A},$ if and only if, for every $i\in J$
\begin{enumerate}
\item $\Delta(e_i)=-\Delta\qty(e_{\sigma(i)})$ and 
\item $\Delta(e_i)(k)=0$ if $k\not\in\qty{i,\sigma(i)} $
\end{enumerate}
\end{theorem}
\begin{proof}
Suppose $\Delta$ is a $\tilde{\sigma}-$derivation on $\mathcal{A}$ and let$\sigma(i)=j,$ then $\tilde{\sigma}(e_i)=e_j.$ If $i\neq j,$ then,
$$\Delta(e_ie_j)=\Delta(0)=0.$$ 
On the other hand,
\begin{align*}
\Delta(e_ie_j)&=\tilde{\sigma}(e_i)\Delta(e_j)+\Delta(e_i)e_j\\
&=e_j\qty(\Delta(e_j)+\Delta(e_i)).
\end{align*}
That is, for every $k\in J,$
\begin{align*}
\Delta(e_ie_j)(k)&=\qty[e_j\qty(\Delta(e_i)+\Delta(e_j))](k)\\
&=\begin{cases} \Delta(e_i)+\Delta(e_j) & \text{ if }k= j\\
0  & \text{ if } k\neq j.
\end{cases}
\end{align*}
Therefore, since $\Delta(e_ie_j)=0,$ then $\Delta(e_i)=-\Delta(e_j).$\par 
Also, for any $k\in J$
\begin{align*}
\Delta(e_i^2)(k)&=\qty(\tilde{\sigma}(e_i)\Delta(e_i)+\Delta(e_i)e_i)(k)\\
&=\Delta(e_i)(e_j+e_i)(k)\\
&=0\qq{ if }k\not\in \{i,j\}.
\end{align*}
Conversely, suppose $\Delta:\mathcal{A}\to \mathcal{A}$ is a $\mathbb{R}-$linear map which satisfies conditions $(1)$ and $(2)$ for some bijection $\sigma:J\to J$ of $J$. We prove that $\Delta$ is $\tilde{\sigma}-$derivation on $\mathcal{A}.$\par 
Suppose $\sigma(i)=j$ and consider the characteristic functions $e_i,e_j$ for $i\neq j.$ Since $\Delta$ is a linear map,
$$\Delta(e_ie_j)=\Delta(0)=0.$$
On the other hand,
\begin{align*}
\tilde{\sigma}(e_i)\Delta(e_j)+\Delta(e_i)e_j&=e_j\Delta(e_j)+\Delta(e_i)e_j\\
&=e_j\qty(\Delta(e_j)+\Delta(e_i))\\
&=0.
\end{align*} 
Therefore 
\begin{equation}\label{eqn11Tumwesigye}
\Delta(e_ie_j)=\tilde{\sigma}(e_i)\Delta(e_j)+\Delta(e_i)e_j
\end{equation} for $i\neq j$ and the same holds for $i=j$. 
The fact that equation \eqref{eqn11Tumwesigye} holds for every $f,g\in \mathcal{A}$ follows from linearity of both $\tilde{\sigma}$ and $\Delta$ and equation \eqref{eqn10Tumwesigye}. Therefore $\Delta$ is a $\tilde{\sigma}-$derivation.
\end{proof}{\flushright{\qed}}
\begin{rem}
It can be seen from Theorem \ref{thm6Tumwesigye} above that, if $\sigma(i)=i$ for all $i\in J,$ then $\Delta=0.$ That is, there are no non trivial derivations on $\mathcal{A}.$
\end{rem}
\subsection{Centralizer for $\mathcal{A}$}
In this section, we give a description of the centralizer  $C(\mathcal{A})$ of $\mathcal{A},$  in the Ore extension $\mathcal{A}[x,\tilde{\sigma}, \Delta]$ and the center of the Ore extension. Since $\mathcal{A}$ is commutative, then by   \cite[Proposition 3.3]{Richter1ABT}, if $\tilde{\sigma}$ is of infinite order, then $\mathcal{A}$ is maximal commutative, that is, $C(\mathcal{A})=\mathcal{A}.$ Therefore we will focus on the case when $\tilde{\sigma}$ is of finite order. We do this for two cases.
\subsubsection{The case $\Delta=0$}
The following Theorem gives the description of the centralizer of $\mathcal{A}$ in the skew-polynomial ring $\mathcal{A}[x,\tilde{\sigma},0].$
\begin{theorem}\label{thm7Tumwesigye}
The centralizer $C(\mathcal{A}),$ of $\mathcal{A}$ in the Ore extension $\mathcal{A}[x,\tilde{\sigma},\Delta]$ is given by
$$
C(\mathcal{A})=\qty{\sum_{k=0}^mf_kx^k\qq{ such that }f_k=0\qq{ on } Sep^k(X)}
$$
where $Sep^k(X)$ is as defined in definition \ref{defn3Tumwesigye}.
\end{theorem}
\begin{proof}
Let $f=\sum\limits_{k=0}^mf_kx^k\in \mathcal{A}[x,\tilde{\sigma},0]$ be an element of degree $m$ which belongs to $C(\mathcal{A}).$ Then $fg=gf$ should hold for every $g\in \mathcal{A}.$\par 
Now, $$gf=g\sum\limits_{k=0}^mf_kx^k=\sum\limits_{k=0}^mgf_kx^k.$$
On the other hand,
$$fg=\qty(\sum\limits_{k=0}^mf_kx^k)g=\sum\limits_{k=0}^mf_k\qty(x^kg)=\sum\limits_{k=0}^mf_k\tilde{\sigma}^k(g)x^k.$$
Therefore, $gf=fg$ if and only if 
$$gf_k=f_k\tilde{\sigma}^k(g)$$ for all $k=0,1,\cdots,m.$ Since $\mathcal{A}$ is commutative, then the above equation holds on $Per^k(X).$ Therefore,
$$
C(\mathcal{A})=\qty{\sum_{k=0}^mf_kx^k\qq{ such that }f_k=0\qq{ on }Sep^k(X)}.
$$
\end{proof}{\flushright{\qed}}
\subsubsection{The case $\Delta \neq 0$}
Now, suppose $\tilde{\sigma}\neq id$ is of order $j\in \mathbb{Z}_{>0},$ that is $\tilde{\sigma}^j=id$ but $\tilde{\sigma}^k\neq id$ for all $k<j$. In the next Theorem, whose proof is similar to Theorem \ref{thm4Tumwesigye}, we state a necessary condition for an element in the Ore extension to belong to the centralizer.
\begin{theorem}\label{thm8Tumwesigye}
Let $\tilde{\sigma}:\mathcal{A}\to \mathcal{A}$ be an automorphism on $\mathcal{A}$. If an element of order $m,$ $\sum\limits_{k=0}^mf_kx^k\in \mathcal{A}[x,\tilde{\sigma},\Delta]$ belongs to the centralizer of $\mathcal{A}$, then $f_m=0\text{ on }Sep^m(X).$
\end{theorem}
\subsection{Center of $\mathcal{A}[x,\tilde{\sigma},\Delta]$ when $\Delta=0$}
In the following section, we give a description of the center of our Ore extension algebra. We will give the description for the case $\Delta=0.$
\begin{theorem}\label{thm9Tumwesigye} 
Then the center of the Ore extension algebra $\mathcal{A}[x,\tilde{\sigma},0]$ is given by
$$Z\qty(\mathcal{A}[x,\tilde{\sigma},0])=\qty{\sum_{k=0}^mf_kx^k:\text{ where }f_k=0 \text{ on }Sep^k(X) \text{ and }\tilde{\sigma}(f_k)=f_k}.$$
\end{theorem}
\begin{proof}
Observe that since $\mathcal{A}$ is not unital, the  proof of Theorem \ref{thm5Tumwesigye} does not work for Theorem \ref{thm9Tumwesigye}, since the element $x\not\in \mathcal{A}[x,\tilde{\sigma},0].$ Therefore, we adopt the following proof. \par 
Denote $\mathcal{A}[x,\tilde{\sigma},\Delta]$ by $R$ and let $f=\sum\limits_{k=0}^mf_kx^k\in Z(R).$ Then $f\in C(\mathcal{A}),$ that is $f_k=0$ on $Sep^k(X).$ Now let $g=\sum\limits_{l=0}^ng_lx^l$ be an arbitrary element in $R.$ Then 
\begin{align*}
fg&=\qty(\sum_{k=0}^mf_kx^k)\qty(\sum_{l=0}^ng_lx^l)\\
&=\qty(\sum_{k,l}f_k(x^kg_l)x^l)\\
&=\qty(\sum_{k,l}f_k\tilde{\sigma}^k(g_l)x^{k+l})
\end{align*}
In the same way, it can be shown that
$$gf=\qty(\sum_{l=0}^ng_lx^l)\qty(\sum_{k=0}^mf_kx^k)=\sum_{k,l}g_l\tilde{\sigma}^l(f_k)x^{k+l}.$$
It follows that $fg=gf$ if and only if
\begin{equation}\label{eqn12Tumwesigye}
f_k\tilde{\sigma}^k(g_l)=g_l\tilde{\sigma}^l(f_k).
\end{equation} for all $k=0,1,\cdots,m$ and all $l=0,1,\cdots,n.$ Now, $f_k=0$ on $Sep^k(X)$ and on $Per^k(X),$ we have $\sigma^k=id.$ Therefore, equation \eqref{eqn12Tumwesigye} holds if and only if 
$$f_kg_l=g_l\tilde{\sigma}^l(f_k)$$
for all $l=0,1,\cdots,n.$ Since $\mathcal{A}$ is commutative, we conclude that \eqref{eqn12Tumwesigye} holds iff $\tilde{\sigma}(f_k)=f_k.$ Therefore 
$$Z\qty(\mathcal{A}[x,\tilde{\sigma},0])=\qty{\sum_{k=0}^mf_kx^k:\text{ where }f_k=0 \text{ on } Sep^k(X)\text{ and }\tilde{\sigma}(f_k)=f_k}.$$
\end{proof}{\flushright{\qed}}
\section{The skew power series ring}\label{sec6Tumwesigye}
As before, we let $X=[n](=\{1,2,\cdots,n\})$ be a finite set and let $\mathcal{A}=\qty{f:X\to \mathbb{R}}$ denote the unital algebra of real-valued functions on $X$ with respect to the usual pointwise operations. Let $\sigma :X\to X$ be a bijection such that $\mathcal{A}$ is invariant under $\sigma$, (that is $\sigma$ is a permutation on $X$) and let $\tilde{\sigma}:\mathcal{A}\to \mathcal{A}$ be the automorphism induced by $\sigma,$ that is 
\begin{equation}\label{eqn13Tumwesigye}
\tilde{\sigma}(f)=f\circ \sigma^{-1}
\end{equation}
for every $f\in \mathcal{A}.$ \par 
Consider the skew ring of formal power series over $\mathcal{A}$, $\mathcal{A}[x;\tilde{\sigma}];$ that is the set 
$$
\qty{\sum_{n=0}^{\infty}f_nx^n\ \text{ such that }f_n\in \mathcal{A}}
$$
with pointwise addition and multiplication determined by the relations
$$xf=\tilde{\sigma}(f)x.$$
That is, if $f=\sum_{n=0}^{\infty}f_nx^n$ and $g=\sum_{n=0}^{\infty}g_nx^n$ are elements of $\mathcal{A}[x;\tilde{\sigma}],$ then
$$
f+g=\sum_{n=0}^{\infty}\qty(f_n+g_n)x^n
$$
and 
$$
fg=\qty(\sum_{n=0}^{\infty}f_nx^n)\qty(\sum_{n=0}^{\infty}g_nx^n)=\sum_{n=0}^{\infty}\qty(\sum_{k=0}^nf_k\tilde{\sigma}^k\qty(g_{n-k}))x^n.
$$
\subsection{Centralizer of $\mathcal{A}$ in the skew power series ring $\mathcal{A}[x;\tilde{\sigma}]$}
In the next Theorem, we give the description of the centralizer of $\mathcal{A}$ in the skew power series ring $\mathcal{A}[x;\tilde{\sigma}].$ 
\begin{theorem}\label{thm10Tumwesigye}
The centralizer $C(\mathcal{A}),$ of $\mathcal{A}$ in the skew power series ring $\mathcal{A}[x;\tilde{\sigma}]$ is given by
$$
C(\mathcal{A})=\qty{\sum_{n\in \mathbb{Z}}f_nx^n\qq{ such that }f_n=0\qq{ on } Sep^n(X)}
$$
where $Sep^k(X)$ is as given in definition \ref{defn3Tumwesigye}.
\end{theorem}
\begin{proof}
Let $f=\sum\limits_{n=0}^{\infty}f_nx^n\in \mathcal{A}[x;\tilde{\sigma}]$ be an element which belongs to $C(\mathcal{A}).$ Then $fg=gf$ should hold for every $g\in \mathcal{A}.$
Now, $$gf=g\sum\limits_{n=0}^{\infty}f_nx^n=\sum\limits_{n=0}^{\infty}gf_nx^n.$$
On the other hand,
$$fg=\qty(\sum\limits_{n=0}^{\infty}f_nx^n)g=\sum\limits_{n=0}^{\infty}f_n\qty(x^ng)=\sum\limits_{n=0}^{\infty}f_n\tilde{\sigma}^n(g)x^n.$$
Therefore, $gf=fg$ if and only if 
$$gf_n=f_n\tilde{\sigma}^n(g)$$ for all $n\in \mathbb{N}.$ Since $\mathcal{A}$ is commutative, then the above equation holds on $Per^n(X).$ Therefore,
$$
C(\mathcal{A})=\qty{\sum_{n\in \mathbb{Z}}f_nx^n\qq{ such that }f_n=0\qq{ on }Sep^n(X)}.
$$
\end{proof}{\flushright{\qed}}
\subsection{The center of the skew power series ring}
The next Theorem gives the description of the center for the skew power series ring $\mathcal{A}[x;\tilde{\sigma}].$
\begin{theorem}\label{thm11Tumwesigye}
The center of the skew power series ring $\mathcal{A}[x;\tilde{\sigma}]$ is given by
$$Z\qty(\mathcal{A}[x;\tilde{\sigma}])=\qty{\sum_{k=0}^{\infty}f_nx^n:\text{ where }f_n=0 \text{ on } Sep^n(X) \text{ and }\tilde{\sigma}(f_n)=f_n}.$$
\end{theorem}
\begin{proof}
Let $f=\sum\limits_{n=0}^{\infty}f_nx^n$ be an element in $Z\qty(\mathcal{A}[x;\tilde{\sigma}]).$ Then $f\in C(\mathcal{A}),$ that is $f_n(x)=0$ for every $x\in Sep^n(x).$ Since $\mathcal{A}[x;\tilde{\sigma}]$ is associative, it is enough to derive conditions under which $xf=fx.$ Now 
$$fx=\qty(\sum_{n=0}^{\infty}f_nx^n)x=\sum_{n=0}^{\infty}f_nx^{n+1}.$$ On the other hand,
\begin{align*}
xf&=x\sum_{n=0}^{\infty}f_nx^n\\
&=\sum_{n=0}^{\infty}xf_n x^n\\
&=\sum_{n=0}^{\infty}\tilde{\sigma}(f_n)x^{n+1}.
\end{align*}
From which we obtain that $xf=fx$ if and only if $\tilde{\sigma}(f_n)=f_n.$ Therefore,
$$Z\qty(\mathcal{A}[x;\tilde{\sigma}])=\qty{\sum_{n=0}^{\infty}f_nx^n:\text{ where }f_n=0 \text{ on }Sep^n(x) \text{ and }\tilde{\sigma}(f_n)=f_n}.$$
\end{proof}{\flushright{\qed}}
\section{The skew-Laurent ring $\mathcal{A}[x,\ x^{-1};\tilde{\sigma}]$}\label{sec7Tumwesigye}
The fact that $\tilde{\sigma}$ is an automorphism of $A$ naturally leads us to the consideration of the skew-Laurent ring $\mathcal{A}[x,\ x^{-1};\tilde{\sigma}].$
\begin{definition}
Let $R$ be a ring and $\sigma$ an automorphism of $R.$ By a skew-Laurent ring $R[x,x^{-1};\sigma]$ we mean that
\begin{enumerate}
\item $R[x,\ x^{-1};\sigma]$ is a ring, containing $R$ as a subring,
\item $x$ is an invertible element of $R[x,\ x^{-1};\sigma]$,
\item $R[x,\ x^{-1};\sigma]$ is a free left $R-$module with basis $\qty{1,x,x^{-1},x^2,x^{-2},\cdots}$,
\item $xr=\sigma(r)x,$ $\qty(\text{and } x^{-1}r=\sigma^{-1}(r)x^{-1})$ for all $r\in R.$
\end{enumerate}
\end{definition}
As before, we let $X=[n](=\{1,2,\cdots,n\})$ and let $\mathcal{A}=\qty{f:X\to \mathbb{R}}$ denote the unital algebra of real-valued functions on $X$ with respect to the usual pointwise operations. Let $\sigma :X\to X$ be a bijection such that $\mathcal{A}$ is invariant under $\sigma$, (that is $\sigma$ is a permutation on $X$) and let $\tilde{\sigma}:\mathcal{A}\to \mathcal{A}$ be the automorphism induced by $\sigma,$ as defined by equation \eqref{eqn13Tumwesigye}.
 \par 
Consider the skew-Laurent ring $\mathcal{A}[x,\ x^{-1};\tilde{\sigma}],$ that is the set 
$$
\qty{\sum_{n\in \mathbb{Z}}f_nx^n\ :\ f_n\in \mathcal{A}\text{ and }f_n=0\text{ for all except finitely many }n}
$$
with pointwise addition and multiplication determined by the relations
$$xf=\tilde{\sigma}(f)x\qquad \text{and }x^{-1}f=\tilde{\sigma}^{-1}x^{-1}.$$
In the next Theorem, we give the description of the centralizer of $\mathcal{A}$ in the skew-Laurent ring $\mathcal{A}[x,\ x^{-1};\tilde{\sigma}].$ 
\begin{theorem}\label{thm12Tumwesigye}
The centralizer of $\mathcal{A}$ in the skew-Laurent  extension $\mathcal{A}[x,\ x^{-1};\tilde{\sigma}]$ is given by
$$
C(\mathcal{A})=\qty{\sum_{n\in \mathbb{Z}}f_nx^n\ :\ f_n=0\qq{ on } Sep^n(X)}
$$
where $Sep^k(X)$ is as given in Definition \ref{defn3Tumwesigye}.
\end{theorem}
\begin{proof}
Let $f=\sum\limits_{n\in \mathbb{Z}}f_nx^n\in \mathcal{A}[x,\ x^{-1};\tilde{\sigma}]$ be an element which belongs to $C(\mathcal{A}).$ Then $fg=gf$ should hold for every $g\in \mathcal{A}.$
Now, $$gf=g\sum\limits_{n\in \mathbb{Z}}f_nx^n=\sum\limits_{n\in \mathbb{Z}}gf_nx^n.$$
On the other hand,
$$fg=\qty(\sum\limits_{n\in \mathbb{Z}}f_nx^n)g=\sum\limits_{n\in \mathbb{Z}}f_n\qty(x^ng)=\sum\limits_{n\in \mathbb{Z}}f_n\tilde{\sigma}^n(g)x^n.$$
Therefore, $gf=fg$ if and only if 
$$gf_n=f_n\tilde{\sigma}^n(g)$$ for all $n\in \mathbb{Z}.$ Since $\mathcal{A}$ is commutative, then the above equation holds on $Per^n(X).$ Therefore,
$$
C(\mathcal{A})=\qty{\sum_{n\in \mathbb{Z}}f_nx^n\ :\ f_n=0\qq{ on }Sep^n(X)}.
$$
\end{proof}
{\flushright{\qed}}
\subsubsection*{Acknowledgement}
This research was supported by the Swedish International Development Cooperation Agency (Sida), International Science Programme (ISP) in Mathematical Sciences (IPMS), Eastern Africa Universities Mathematics Programme (EAUMP).  Alex Tumwesigye is also grateful to the research environment Mathematics and Applied Mathematics (MAM), Division of Applied Mathematics, M\"alardalen University for providing an excellent and inspiring environment for research education and research.
%\end{acknowledgement}
%

%\input{referenc}
\end{document}